\documentclass[12pt,reqno]{amsart}
\usepackage{amsmath, amsfonts, amssymb, amsthm, hyperref,cite}
\usepackage{bm}
\allowdisplaybreaks[4]
\def\l{\left}
\def\r{\right}
\def\bg{\bigg}
\def\({\bg(}
\def\){\bg)}
\def\t{\text}
\def\f{\frac}

\def\eq{\equiv}

\def\Z{\mathbb Z}
\def\C{\mathbb C}
\def\N{\mathbb N}
\def\Q{\mathbb Q}

\def\1{{\bf 1}}

\def\pmod #1{\ ({\rm{mod}}\ #1)}

\def\<{\langle}
\def\>{\rangle}
\theoremstyle{plain}
\newtheorem{theorem}{Theorem}[section]
\newtheorem{lemma}{Lemma}
\newtheorem{corollary}{Corollary}
\newtheorem{conjecture}{Conjecture}

\theoremstyle{definition}

\newtheorem*{Acks}{Acknowledgments}
\theoremstyle{remark}
\newtheorem{remark}{Remark}

\begin{document}
\hbox{Accepted by J. Math. Anal. Appl.}
\title[Proof of some conjectural hypergeometric supercongruences]
{Proof of some conjectural hypergeometric\\supercongruences via curious identities}
\author[Chen Wang]{Chen Wang}
\address[Chen Wang]{Department of Applied Mathematics, Nanjing Forestry University, Nanjing 210037, People's Republic of China}
\email{cwang@smail.nju.edu.cn}

\author[Zhi-Wei Sun]{Zhi-Wei Sun}
\address[Zhi-Wei Sun]{Department of Mathematics, Nanjing
University, Nanjing 210093, People's Republic of China}
\email{zwsun@nju.edu.cn}

\subjclass[2010]{Primary 33C20, 11B65; Secondary 05A19, 11A07, 33E50}
\keywords{Truncated hypergeometric series, $p$-adic Gamma function, congruences, binomial coefficients}

\begin{abstract}
In this paper, we prove several supercongruences conjectured by Z.-W. Sun ten years ago via certain strange hypergeometric identities. For example, for any prime $p>3$, we show that $$\sum_{k=0}^{p-1}\frac{\binom{4k}{2k+1}\binom{2k}k}{48^k}\equiv0\pmod{p^2},$$
and
$$
\sum_{k=0}^{p-1}\frac{\binom{2k}{k}\binom{3k}{k}}{24^k}\equiv\begin{cases}\binom{(2p-2)/3}{(p-1)/3}\pmod{p^2}\ &\mbox{if}\  p\equiv1\pmod{3},\vspace{2mm}\\
p/\binom{(2p+2)/3}{(p+1)/3}\pmod{p^2}\ &\mbox{if}\ p\equiv2\pmod{3}.\end{cases}
$$
We also obtain some other results of such types.
\end{abstract}
\maketitle

\section{Introduction}
\setcounter{lemma}{0}
\setcounter{theorem}{0}
\setcounter{equation}{0}
\setcounter{conjecture}{0}
\setcounter{remark}{0}
\setcounter{corollary}{0}

For $n,r\in\N=\{0,1,2,\ldots\}$ and $\alpha_0,\ldots,\alpha_r,\beta_1,\ldots,\beta_r,z\in\C$ the truncated hypergeometric series ${}_{r+1}F_r$ are defined by
$$
{}_{r+1}F_r\bigg[\begin{matrix}\alpha_0&\alpha_1&\cdots&\alpha_r\\ &\beta_1&\cdots&\beta_r\end{matrix}\bigg|\ z\bigg]_n:=\sum_{k=0}^{n}\f{(\alpha_0)_k\cdots(\alpha_r)_k}{(\beta_1)_k\cdots(\beta_r)_k}\cdot\f{z^k}{k!},
$$
where $(\alpha)_k=\alpha(\alpha+1)\cdots(\alpha+k-1)$ is the Pochhammer symbol (rising factorial). Since $(-\alpha)_k/(1)_k=(-1)^k\binom{\alpha}{k}$, sometimes we may write the truncated hypergeometric series as sums involving products of binomial coefficients. In the past decades, supercongruences involving truncated hypergeometric series have been widely studied (cf. for example, \cite{Guo,GuoLiu,Liu,Liu2021,LR,MP,Mortenson1,Mortenson2,ZHSun2014,Sun2011,Sun2013,Wang,WangPan}).

Via the $p$-adic Gamma function and the Gross-Koblitz formula, E. Mortenson \cite{Mortenson1,Mortenson2} proved that for any prime $p>3$ we have
\begin{equation}\label{mortenson's}
\begin{gathered}
{}_2F_1\bigg[\begin{matrix}\f12&\f12\\ &1\end{matrix}\bigg|\ 1\bigg]_{p-1}\eq\l(\f{-1}{p}\r)\pmod{p^2},\quad{}_2F_1\bigg[\begin{matrix}\f13&\f23\\ &1\end{matrix}\bigg|\ 1\bigg]_{p-1}\eq\l(\f{-3}{p}\r)\pmod{p^2},\\
{}_2F_1\bigg[\begin{matrix}\f14&\f34\\ &1\end{matrix}\bigg|\ 1\bigg]_{p-1}\eq\l(\f{-2}{p}\r)\pmod{p^2},\quad{}_2F_1\bigg[\begin{matrix}\f16&\f56\\ &1\end{matrix}\bigg|\ 1\bigg]_{p-1}\eq\l(\f{-1}{p}\r)\pmod{p^2},
\end{gathered}
\end{equation}
where $(\f{\cdot}{p})$ denotes the Legendre symbol. Actually, these congruences were first conjectured in \cite{RV} motivated by hypergeometric families of Calabi-Yau manifolds. For any prime $p>3$, Z.-W. Sun \cite{Sun2011} showed further that
\begin{gather}
\label{sun's1}\sum_{k=0}^{(p-1)/2}\f{\binom{2k}{k}^2}{16^k}\eq(-1)^{(p-1)/2}+p^2E_{p-3}\pmod{p^3},\\
\label{sun's2}\sum_{p/2<k<p}\f{\binom{2k}{k}^2}{16^k}\eq-2p^2E_{p-3}\pmod{p^3},
\end{gather}
where $E_{p-3}$ is the $(p-3)$th Euler number. In fact, \eqref{sun's1} and \eqref{sun's2} are refinements of the first congruence in \eqref{mortenson's}. To see this, we note that $(\f12)_k/(1)_k=\binom{2k}{k}/4^k$ for any $k\in\N$. Z.-W. Sun \cite{Sun2013}
also gave some extensions of \eqref{mortenson's}. In 2014, Z.-H. Sun \cite{ZHSun2014} found that the congruences in \eqref{mortenson's} can be extended to a unified form. For any odd prime $p$ let $\Z_p$ denote the ring of all $p$-adic integers. For any $\alpha\in\Z_p$, we use $\<\alpha\>_p$ to denote the least nonnegative residue of $\alpha$ modulo $p$, i.e., the unique integer lying in $\{0,1,\ldots,p-1\}$ such that $\<\alpha\>_p\eq\alpha\pmod{p}$. Z.-H. Sun \cite{ZHSun2014} proved that for any $\alpha\in\Q\cap\Z_p$ we have
\begin{equation}\label{sunzh's}
{}_2F_1\bigg[\begin{matrix}\alpha&1-\alpha\\ &1\end{matrix}\bigg|\ 1\bigg]_{p-1}\eq(-1)^{\<-\alpha\>_p}\pmod{p^2}.
\end{equation}
It is easy to see that the congruences in \eqref{mortenson's} are special cases of \eqref{sunzh's}.

It is worth noting that Mortenson's congruences in \eqref{mortenson's} all concern the hypergeometric series with variable $z=1$. In this paper, we shall confirm several congruences involving hypergeometric series with variable $z\neq1$, as conjectured by Z.-W. Sun.

It is rare that a sum of the form $\sum_{k=0}^{p-1}a_k$ is always congruent to $0$ modulo $p^2$ for any prime $p>3$.
The only known example we can recall is Wolstenholme's congruence $H_{p-1}\eq0\pmod{p^2}$ (cf. \cite{W}) for any prime $p>3$, where $H_n$
denotes the harmonic number $\sum_{0<k\le n}1/k$. Nevertheless, we establish the following curious result which was first conjectured by Sun \cite[Conjecture 5.14(i)]{Sun2011}.

 \begin{theorem}\label{48theorem}
Let $p>3$ be a prime. Then
\begin{equation}\label{48theq}
\sum_{k=0}^{p-1}\f{\binom{4k}{2k+1}\binom{2k}{k}}{48^k}\eq0\pmod{p^2}.
\end{equation}
\end{theorem}

\begin{remark}
Sun \cite[Conjecture 17]{Sun2019} conjectured further that for any prime $p>3$ we have
\begin{equation}\label{48conj}
\sum_{k=0}^{p-1}\f{\binom{4k}{2k+1}\binom{2k}{k}}{48^k}\eq\f{5}{12}p^2B_{p-2}\l(\f13\r)\pmod{p^3}
\end{equation}
and
\begin{equation}\label{48conjdual}
p^2\sum_{k=1}^{p-1}\f{48^k}{k(2k-1)\binom{4k}{2k}\binom{2k}{k}}\eq4\l(\f{p}3\r)+4p\pmod{p^2},
\end{equation}
where $B_{p-2}(x)$ is the Bernoulli polynomial of degree $p-2$.  It is worth mentioning that \eqref{48conjdual} is related to Sun's conjectural identity
\begin{equation}\label{48series}
\sum_{k=1}^{\infty}\f{48^k}{k(2k-1)\binom{4k}{2k}\binom{2k}{k}}=\f{15}2\sum_{k=1}^{\infty}\f{(\f{k}3)}{k^2}
\end{equation}
(cf. \cite[(1.23)]{Sun2015}) and Sun would like to offer $\$480$ as the prize for the first proof of \eqref{48series}.
\end{remark}

 Our second theorem concerns a variant of the second congruence in \eqref{mortenson's} and confirms a conjecture of Z.-W. Sun in \cite[Conjecture 5.13]{Sun2011} and \cite[Conjecture 16(\rm{i})]{Sun2019}.

\begin{theorem}\label{main1}
Let $p>3$ be a prime. Then
\begin{equation}\label{main1eq}
\sum_{k=0}^{p-1}\frac{\binom{2k}{k}\binom{3k}{k}}{24^k}\equiv\begin{cases}\binom{(2p-2)/3}{(p-1)/3}\pmod{p^2}\ &{\rm if}\  p\equiv1\pmod{3},\vspace{2mm}\\
 p/\binom{(2p+2)/3}{(p+1)/3}\pmod{p^2}\ &{\rm if}\ p\equiv2\pmod{3}.\end{cases}
\end{equation}
\end{theorem}

\begin{remark} By Sun \cite[(1.20)]{Sun2013}, for any prime $p>3$ we have
\begin{equation*}
\sum_{k=0}^{p-1}\f{\binom{2k}{k}\binom{3k}{k}}{24^k}\eq\l(\f{p}{3}\r)\sum_{k=0}^{p-1}\f{\binom{2k}{k}\binom{3k}{k}}{(-216)^k}\pmod{p^2}.
\end{equation*}
\end{remark}

Theorem 3.2 of Sun \cite{Sun2014} with $x=y=-z$ and $a=1$ gives the following $p$-adic analogue of the Clausen identity (cf. \cite[p. 116]{AAR}):
\begin{equation}\label{clausen}
\l({}_2F_1\bigg[\begin{matrix}\alpha&1-\alpha\\&1\end{matrix}\bigg|\ z\bigg]_{p-1}\r)^2\eq{}_3F_2\bigg[\begin{matrix}\alpha&1-\alpha&\f12\\&1&1\end{matrix}\bigg|\ 4z(1-z)\bigg]_{p-1}\pmod{p^2}
\end{equation}
for any odd prime $p$ and $\alpha,z\in\Z_p$; this was given by Z.-H. Sun in the cases $\alpha=1/3,1/4,1/6$ (cf. \cite{ZHSun2013a,ZHSun2013b,ZHSun2013c}). Applying \eqref{clausen} with $\alpha=1/3$ and $z=9/8$ and noting that $(1/3)_k(2/3)_k/(1)_k^2=\binom{2k}{k}\binom{3k}{k}/27^k$ we obtain that
\begin{equation}\label{square}
\l(\sum_{k=0}^{p-1}\f{\binom{2k}{k}\binom{3k}{k}}{24^k}\r)^2\eq\sum_{k=0}^{p-1}\frac{\binom{2k}{k}^2\binom{3k}{k}}{(-192)^k}\pmod{p^2}.
\end{equation}
It is known (cf. \cite{BEW,C}) that for any prime $p\eq1\pmod{3}$ with $4p=x^2+27y^2\ (x,y\in\Z)$ we have
\begin{equation}\label{quadratic_form1}
\binom{(2p-2)/3}{(p-1)/3}\eq\l(\f{x}{3}\r)\l(\f{p}{x}-x\r)\pmod{p^2}.
\end{equation}
Combining Theorem \ref{main1}, \eqref{square} and \eqref{quadratic_form1} we immediately obtain the following result which was conjectured by Z.-W. Sun in \cite[Conjecture 5.6]{Sun2011} and \cite[Conjecture 24(\rm{i})]{Sun2019} and partially proved by Z.-H. Sun \cite[Theorem 4.2]{ZHSun2013a}.

\begin{corollary}
Let $p>3$ be a prime. Then
\begin{equation*}
\sum_{k=0}^{p-1}\frac{\binom{2k}{k}^2\binom{3k}{k}}{(-192)^k}\equiv\begin{cases} x^2-2p\pmod{p^2}\ &{\rm if}\ p\equiv1\pmod{3}\ \& \ 4p=x^2+27y^2\ (x,y\in\Z),\vspace{2mm}\\
 0\pmod{p^2}\ &{\rm if}\ p\equiv2\pmod{3}.\end{cases}
\end{equation*}
\end{corollary}

The next result gives a companion of Theorem \ref{main1}.
\begin{theorem}\label{main1'}
Let $p>3$ be a prime. Then
\begin{equation}\label{main1'eq}
\sum_{k=0}^{p-1}\f{(k+1)\binom{3k}{k}\binom{2k}{k}}{24^k}\eq\begin{cases}p/\binom{(2p-2)/3}{(p-1)/3}\pmod{p^2}\ &{\rm if}\ p\eq1\pmod{3},\vspace{2mm}\\
-(p+1)\binom{(2p+2)/3}{(p+1)/3}\pmod{p^2}\ &{\rm if}\ p\eq2\pmod{3}.
\end{cases}
\end{equation}
\end{theorem}

Combining Theorems \ref{main1} and \ref{main1'} we confirm the following two conjectures of Sun \cite[Conjecture 5.13]{Sun2011}.
\begin{corollary}Let $p>3$ be a prime. Then
\begin{equation}\label{den}
\sum_{k=0}^{p-1}\f{\binom{2k}{k}\binom{3k}{k}}{(k+1)24^k}\eq\f12\binom{2(p-\l(\f{p}3\r))/3}{(p-\l(\f{p}3\r))/3}\pmod{p}.
\end{equation}
When $p\eq1\pmod{3}$ and $4p=x^2+27y^2$ with $x,y\in\Z$ and $x\eq2\pmod{3}$, we have
\begin{equation}\label{x}
\sum_{k=0}^{p-1}\f{k+2}{24^k}\binom{2k}{k}\binom{3k}{k}\eq x\pmod{p^2}.
\end{equation}
\end{corollary}

\begin{remark}
To obtain \eqref{den}, we note the following congruence relation obtained by Sun \cite{Sun2013}:
$$
\sum_{k=0}^{p-1}\f{\binom{2k}{k}\binom{3k}{k}}{(k+1)m^k}\eq p+\f{m-27}{6}\sum_{k=0}^{p-1}\f{k\binom{2k}{k}\binom{3k}{k}}{m^k}\pmod{p^2},
$$
where $p>3$ is a prime and $m$ is an integer with $p\nmid m$. To get \eqref{x} we only need to substitute \eqref{quadratic_form1} into \eqref{main1eq} and \eqref{main1'eq}.
\end{remark}

The proofs of the above three theorems depend on some new hypergeometric identities motivated by the strange identities obtained by S. B. Ekhad \cite{Ekhad} and the Pfaff transformation (cf. \cite[p. 68]{AAR}); they will be given in Sections 2--4. We can also prove some other results similar to Theorems \ref{48theorem}--\ref{main1'} in the same way, however, we will not give the detailed proofs of them; we shall list them and sketch their proofs in the last section.

\section{Proof of Theorem \ref{48theorem}}
\setcounter{lemma}{0}
\setcounter{theorem}{0}
\setcounter{equation}{0}
\setcounter{conjecture}{0}
\setcounter{remark}{0}
\setcounter{corollary}{0}

Let us recall the concept of the $p$-adic Gamma function introduced by Y. Morita \cite{Morita} as a $p$-adic analogue of the classical Gamma function. For each integer $n\geq1$, define the $p$-adic Gamma function
$$
\Gamma_p(n):=(-1)^n\prod_{\substack{1\leq k<n\\ p\nmid k}}k.
$$
Morever, set $\Gamma_p(0)=1$. It is clear that the definition of $\Gamma_p$ can be extended to $\Z_p$ since $\N$ is a dense subset of $\Z_p$ with respect to $p$-adic norm $|\cdot|_p$. This means that for all $x\in\Z_p$ we can define
$$
\Gamma_p(x):=\lim_{\substack{n\in\N\\ |x-n|_p\to0}}\Gamma_p(n).
$$
The reader is referred to \cite{Morita,Robert} for some properties of the $p$-adic Gamma functions. It is known (cf. \cite[p. 369]{Robert}) that for any $x\in\Z_p$ we have
\begin{equation}\label{padicgamma1}
\Gamma_p(x)\Gamma_p(1-x)=(-1)^{p-\<-x\>_p}
\end{equation}
and
\begin{equation}\label{padicgamma2}
\f{\Gamma_p(x+1)}{\Gamma_p(x)}=\begin{cases}-x\ &if\ p\nmid x,\vspace{2mm}\\
-1\ &if\ p\mid x.\end{cases}
\end{equation}

The following lemma gives a $p$-adic expansion of $\Gamma_p$.
\begin{lemma}\label{lem}
For any prime $p>3$ and $\alpha,t\in\Z_p$ we have
\begin{equation}\label{lemeq}
\Gamma_p(\alpha+tp)\eq\Gamma_p(\alpha)\l(1+tp(\Gamma_p'(0)+H_{p-1-\<-\alpha\>_p})\r)\pmod{p^2},
\end{equation}
where $H_n=\sum_{k=1}^n1/k$ is the $n$th harmonic number.
\end{lemma}
\begin{proof}
It is known (cf. \cite[Theorem 14]{LR}) that
$$
\Gamma_p(\alpha+tp)\eq\Gamma_p(\alpha)(1+tp\Gamma_p'(\alpha))\pmod{p^2}.
$$
By \cite[Lemma 2.4]{WangPan} we have
$$
\f{\Gamma_p'(\alpha)}{\Gamma_p(\alpha)}\eq\Gamma_p'(0)+H_{p-1-\<-\alpha\>_p}\pmod{p}.
$$
Combining the above we immediately get the desired result.
\end{proof}

The following hypergeometric identity plays a key role in the proof of Theorem \ref{48theorem}.

\begin{lemma}\label{48lem}
For any nonnegative integer $n$ and $\delta\in\{0,1\}$, we have
\begin{equation}\label{48lemeq1}
\sum_{k=0}^{n}\f{(4n+2\delta-k+1)(-n)_k(\f12-\delta-n)_k}{(2n+\delta+k+1)(1)_k(-4n-2\delta)_k}\l(\f43\r)^k=\l(\f34\r)^{2n+\delta}\f{\Gamma(\f12)\Gamma(2n+1+\delta)}{\Gamma(2n+\delta+\f12)}.
\end{equation}
\end{lemma}
\begin{proof}
Denote the left-hand side of \eqref{48lemeq1} by $S(n)$. Via Zeilberger's algorithm (cf. \cite{PWZ}) in \verb"Mathematica", we find that
$$
9(n+1)(2n+2\delta+1)S(n)-2(4n+2\delta+1)(4n+2\delta+3)S(n+1)=0
$$
for any $n\in\N$ and $\delta\in\{0,1\}$. Then the identities can be verified by induction on $n$.
\end{proof}

For a prime $p$ and an integer $a\eq0\pmod{p}$, we use $q_p(a)$ to denote the Fermat quotient $(a^{p-1}-1)/p$.
\begin{lemma}[Lehmer \cite{L}]\label{lem2} For any prime $p>3$, we have the following congruences
$$
H_{\lfloor p/2\rfloor}\eq-2q_p(2),\ H_{\lfloor p/3\rfloor}\eq-\f32q_p(3),\ H_{\lfloor p/4\rfloor}\eq-3q_p(2),\ H_{\lfloor p/6\rfloor}\eq-2q_p(2)-\f32q_p(3)
$$
modulo $p$.
\end{lemma}

\medskip
\noindent{\it Proof of Theorem \ref{48theorem}}. Note that $(1/4)_k(3/4)_k/(1)_k^2=\binom{4k}{2k}\binom{2k}{k}/64^k$. Thus we can rewrite \eqref{48theq} as follows:
$$
\sum_{k=0}^{p-1}\f{2k(\f14)_k(\f34)_k}{(2k+1)(1)_k^2}\l(\f43\r)^k\eq0\pmod{p^2}.
$$
Let
$$
f_k(x):=\f{(x-k)(\f{1-x}4)_k(\f{3-x}4)_k}{(\f{x}2+k+\f12)(1-x)_k(1)_k}\l(\f43\r)^k
$$
and
$$
f(x):=\sum_{\substack{k=0\\ k\neq (p-1)/2}}^{p-1}f_k(x).
$$
Clearly, for any $t\in\Z_p$ we have
$$
f(tp)\eq f(0)+tpf'(0)\pmod{p^2},
$$
where $f'(x)$ stands for the derivative of $f(x)$. Thus we can easily obtain that
\begin{equation}\label{f0}
f(0)\eq\f12(3f(p)-f(3p))\pmod{p^2}.
\end{equation}
On the other hand, it is easy to see that
\begin{equation}\label{f0'}
f(0)=-\sum_{k=0}^{p-1}\f{2k(\f14)_k(\f34)_k}{(2k+1)(1)_k^2}\l(\f43\r)^k+\varepsilon,
\end{equation}
where
$$
\varepsilon:=\f{(p-1)(\f14)_{(p-1)/2}(\f34)_{(p-1)/2}}{p(1)_{(p-1)/2}^2}\l(\f43\r)^{(p-1)/2}.
$$
Combining \eqref{f0} and \eqref{f0'} we arrive at
\begin{equation}\label{48key}
\sum_{k=0}^{p-1}\f{2k(\f14)_k(\f34)_k}{(2k+1)(1)_k^2}\l(\f43\r)^k\eq \varepsilon-\f32f(p)+\f12f(3p)\pmod{p^2}.
\end{equation}

We first consider $\varepsilon$ modulo $p^2$. Clearly,
$$
\f{\Gamma(-\f14+\f{p}2)\Gamma(\f14+\f{p}2)}{\Gamma(\f14)\Gamma(\f34)}=\f{p\Gamma_p(-\f14+\f{p}2)\Gamma_p(\f14+\f{p}2)}{4\Gamma_p(\f14)\Gamma_p(\f34)}.
$$
Thus by Lemma \ref{lem} we have
\begin{align*}
\varepsilon&=\f{(p-1)\Gamma(-\f14+\f{p}2)\Gamma(\f14+\f{p}2)\Gamma(1)^2}{p\Gamma(\f14)\Gamma(\f34)\Gamma(\f12+\f{p}2)^2}\l(\f43\r)^{(p-1)/2}\\
&=\f{(p-1)\f{p}4\Gamma_p(-\f14+\f{p}2)\Gamma_p(\f14+\f{p}2)\Gamma_p(1)^2}{p\Gamma_p(\f14)\Gamma_p(\f34)\Gamma_p(\f12+\f{p}2)^2}\l(\f43\r)^{(p-1)/2}\\
&\eq\f{(-p-1)}{\Gamma_p(\f12)^2}\l(1+pH_{\lfloor p/4\rfloor}-pH_{(p-1)/2}\r)\l(\f43\r)^{(p-1)/2}\pmod{p^2}.
\end{align*}
Note that
$$
\l(\f43\r)^{(p-1)/2}=3^{(p-1)/2}\l(\f23\r)^{p-1}\eq 3^{(p-1)/2}(1+p\,q_p(2)-p\,q_p(3))\pmod{p^2}.
$$
Then by \eqref{padicgamma1} and Lemma \ref{lem2} we immediately obtain
\begin{equation}\label{delta}
\varepsilon\eq(-3)^{(p-1)/2}(1+p-p\,q_p(3))\pmod{p^2}.
\end{equation}

Now we evaluate $f(p)$ modulo $p^2$. Obviously, $f_{(p-1)/2}(p)=0$ since
$$
\l(-\l\lfloor\f{p}4\r\rfloor\r)_{(p-1)/2}=0.
$$
Taking $n=\lfloor p/4\rfloor$ and $\delta=(1-(\f{-1}{p}))/2$
in Lemma \ref{48lem} and using \eqref{padicgamma1}, Lemmas \ref{lem} and \ref{lem2} we obtain
\begin{equation}\label{fp}
\begin{aligned}
f(p)&=f(p)+f_{(p-1)/2}(p)=\l(\f34\r)^{(p-1)/2}\f{\Gamma(\f12)\Gamma(\f{p+1}2)}{\Gamma(\f{p}2)}=-\f{3^{(p-1)/2}}{2^{p-1}}\f{\Gamma_p(\f12)\Gamma_p(\f{p+1}2)}{\Gamma_p(\f{p}2)}\\
&\eq-3^{(p-1)/2}\Gamma_p\l(\f12\r)^2\l(1+\f{p}2H_{(p-1)/2}\r)(1-p\,q_p(2))\\
&\eq(-3)^{(p-1)/2}(1-2p\,q_p(2))\pmod{p^2}.
\end{aligned}
\end{equation}

Finally, we consider $f(3p)$ modulo $p^2$. Taking $n=\lfloor3p/4\rfloor$ and $\delta=(1+(\f{-1}{p}))/2$
in Lemma \ref{48lem} we have
\begin{equation}\label{f3pkey}
f(3p)+f_{(p-1)/2}(3p)=\l(\f34\r)^{(3p-1)/2}\f{\Gamma(\f12)\Gamma(\f{3p+1}2)}{\Gamma(\f{3p}2)}.
\end{equation}
Similarly as in \eqref{fp} we deduce that
\begin{equation}\label{f3pkey1}
\begin{aligned}
\l(\f34\r)^{(3p-1)/2}\f{\Gamma(\f12)\Gamma(\f{3p+1}2)}{\Gamma(\f{3p}2)}&=\l(\f34\r)^{(3p-1)/2}\f{p\Gamma_p(\f12)\Gamma_p(\f{3p+1}2)}{\f{p}2\Gamma_p(\f{3p}2)\Gamma_p(1)}\\
&\eq\f32(-3)^{(p-1)/2}(1+p\,q_p(3)-6p\,q_p(2))\pmod{p^2}.
\end{aligned}
\end{equation}
Observe that
\begin{align*}
f_{(p-1)/2}(3p)&=\f{(5p+1)\Gamma(\f14-\f{p}4)\Gamma(-\f14-\f{p}4)\Gamma(1)\Gamma(1-3p)}{4p\Gamma(\f34-\f{3p}4)\Gamma(\f14-\f{3p}4)\Gamma(\f12+\f{p}2)\Gamma(\f12-\f{5p}2)}\l(\f43\r)^{(p-1)/2}\\
&=\f{-\f{p}2(5p+1)\Gamma_p(\f14-\f{p}4)\Gamma_p(-\f14-\f{p}4)\Gamma_p(1)\Gamma_p(1-3p)}{4p\Gamma_p(\f34-\f{3p}4)\Gamma_p(\f14-\f{3p}4)\Gamma_p(\f12+\f{p}2)\Gamma_p(\f12-\f{5p}2)}\l(\f43\r)^{(p-1)/2},
\end{align*}
where we note that
$$
\f{\Gamma(\f14-\f{p}4)\Gamma(-\f14-\f{p}4)}{\Gamma(\f34-\f{3p}4)\Gamma(\f14-\f{3p}4)}=-\f{p\Gamma_p(\f14-\f{p}4)\Gamma_p(-\f14-\f{p}4)}{2\Gamma_p(\f34-\f{3p}4)\Gamma_p(\f14-\f{3p}4)}.
$$
Furthermore, in view of \eqref{padicgamma1}, Lemmas \ref{lem} and \ref{lem2}, we arrive at
\begin{equation}\label{f3pkey2}
f_{(p-1)/2}(3p)\eq\f{(-3)^{(p-1)/2}}{2}(1+4p-6p\,q_p(2)-p\,q_p(3))\pmod{p^2}.
\end{equation}

Now combining \eqref{48key}--\eqref{f3pkey2} we finally obtain
$$
\sum_{k=0}^{p-1}\f{2k(\f14)_k(\f34)_k}{(2k+1)(1)_k^2}\l(\f43\r)^k\eq0\pmod{p^2}.
$$
This completes the proof.\qed

\section{Proof of Theorem \ref{main1}}
\setcounter{lemma}{0}
\setcounter{theorem}{0}
\setcounter{equation}{0}
\setcounter{conjecture}{0}
\setcounter{remark}{0}
\setcounter{corollary}{0}

\begin{lemma}\label{lem1}
For any nonnegative integer $n$ we have
\begin{equation}\label{lem1eq1}
{}_2F_1\bigg[\begin{matrix}-n&\f13-n\\&-3n\end{matrix}\bigg|\ \f98\bigg]_n=\f{(\f12)_n}{2^n(\f13)_n}
\end{equation}
and
\begin{equation}\label{lem1eq2}
{}_2F_1\bigg[\begin{matrix}-n&-\f13-n\\&-3n-1\end{matrix}\bigg|\ \f98\bigg]_n=\f{(\f56)_n}{2^n(\f23)_n}.
\end{equation}
\end{lemma}
\begin{proof}
Denote the left-hand sides of \eqref{lem1eq1} and \eqref{lem1eq2} by $F(n)$ and $G(n)$ respectively. By applying the Zeilberger algorithm in \verb"Mathematica", we find that
$$
-3(2n+1)F(n)+4(3n+1)F(n+1)=0
$$
and
$$
-(6n+5)G(n)+4(3n+2)G(n+1)=0.
$$
Then Lemma \ref{lem1} follows by induction on $n$.
\end{proof}

The next lemma is a $p$-adic analogue of the classical Gauss multiplication formula.
\begin{lemma}[Robert {\cite[p. 371]{Robert}}]\label{Gauss}
Let $p$ be an odd prime. Then for any $x\in\Z_p$ and $m\in\Z^{+}$ we have
\begin{equation}\label{Gausseq}
\prod_{0\leq j<m}\Gamma_p\l(x+\f{j}{m}\r)=\Gamma_p(mx)\prod_{0\leq j<m}\Gamma_p\l(\f{j}{m}\r)m^{((1-p)mx+\<-mx\>_p)/p}.
\end{equation}
\end{lemma}

\noindent{\it Proof of Theorem \ref{main1}}. Assume that $p\eq1\pmod{3}$. Clearly, for any $\alpha,t\in\Z_p$ we have
\begin{equation}\label{poch}
(\alpha+tp)_k\eq(\alpha)_k\l(1+tp\sum_{j=0}^{k-1}\f{1}{\alpha+j}\r)\pmod{p^2}.
\end{equation}
Therefore, it is easy to check that
\begin{align*}
&{}_2F_1\bigg[\begin{matrix}\f{1-p}{3}&\f{2-p}{3}\\&1-p\end{matrix}\bigg|\ \f98\bigg]_{p-1}
\\&\qquad\eq\sum_{k=0}^{p-1}\f{(\f13)_k(\f23)_k}{(1)_k^2}\l(\f98\r)^k\l(1-p\sum_{j=0}^{k-1}\f{1}{3j+1}-p\sum_{j=0}^{k-1}\f{1}{3j+2}+pH_k\r)\pmod{p^2}
\end{align*}
and
\begin{align*}
&{}_2F_1\bigg[\begin{matrix}\f{2-2p}{3}&\f{1-2p}{3}\\&1-2p\end{matrix}\bigg|\ \f98\bigg]_{p-1}
\\&\qquad\eq\sum_{k=0}^{p-1}\f{(\f13)_k(\f23)_k}{(1)_k^2}\l(\f98\r)^k\l(1-2p\sum_{j=0}^{k-1}\f{1}{3j+1}-2p\sum_{j=0}^{k-1}\f{1}{3j+2}+2pH_k\r)\pmod{p^2}.
\end{align*}
Combining the above two congruences we arrive at
$$
\sum_{k=0}^{p-1}\f{(\f13)_k(\f23)_k}{(1)_k^2}\l(\f98\r)^k\eq2{}_2F_1\bigg[\begin{matrix}\f{1-p}{3}&\f{2-p}{3}\\&1-p\end{matrix}\bigg|\ \f98\bigg]_{p-1}-{}_2F_1\bigg[\begin{matrix}\f{2-2p}{3}&\f{1-2p}{3}\\&1-2p\end{matrix}\bigg|\ \f98\bigg]_{p-1}\pmod{p^2}.
$$
Moreover, letting $n=(p-1)/3$ in \eqref{lem1eq1} and $n=(2p-2)/3$ in \eqref{lem1eq2} we obtain
\begin{equation}\label{sigma12}
\sum_{k=0}^{p-1}\f{(\f13)_k(\f23)_k}{(1)_k^2}\l(\f98\r)^k\eq\sigma_1-\sigma_2\pmod{p^2},
\end{equation}
where
$$
\sigma_1:=2^{(4-p)/3}\f{(\f12)_{(p-1)/3}}{(\f13)_{(p-1)/3}}\quad\t{and}\quad\sigma_2:=2^{(2-2p)/3}\f{(\f56)_{(2p-2)/3}}{(\f23)_{(2p-2)/3}}.
$$

We first compute $\sigma_1$ modulo $p^2$. Note that
\begin{equation}\label{key}
\f{(\f12)_{(p-1)/3}}{(1)_{(p-1)/3}}=\f{\binom{(2p-2)/3}{(p-1)/3}}{4^{(p-1)/3}}=4^{(1-p)/3}\f{\Gamma(\f13+\f{2p}3)}{\Gamma(\f23+\f{p}3)^2}.
\end{equation}
Therefore, by Lemmas \ref{lem} and \ref{lem2} we have
\begin{equation}\label{sigma1}
\begin{aligned}
\sigma_1&=2^{(4-p)/3}\f{(\f12)_{(p-1)/3}(1)_{(p-1)/3}}{(1)_{(p-1)/3}(\f13)_{(p-1)/3}}\\
&\eq\f{\Gamma_p(\f13+\f{2p}3)\Gamma_p(\f13)}{\Gamma_p(\f23+\f{p}3)\Gamma_p(\f{p}3)}(2-2p\,q_p(2))\\
&\eq\f{\Gamma_p(\f13)^2}{\Gamma_p(\f23)}(2-2p\,q_p(2))\l(1+\f{2p}{3}H_{(2p-2)/3}-\f{p}{3}H_{(p-1)/3}-\f{p}{3}H_{p-1}\r)\\
&\eq\Gamma_p\l(\f13\r)^3(p\,q_p(3)+2p\,q_p(2)-2)\pmod{p^2},
\end{aligned}
\end{equation}
where in the last step we note that $\Gamma_p(1/3)\Gamma_p(2/3)=(-1)^{1+(2p-2)/3}=-1$ by \eqref{padicgamma1} and $H_{p-1-k}\eq H_k\pmod{p}$ for any $k\in\{0,1,\ldots,p-1\}$.

We now evaluate $\sigma_2$ modulo $p^2$. Similarly, by \eqref{key}, Lemmas \ref{lem} and \ref{lem2} we have
\begin{align*}
\sigma_2&=2^{(2-2p)/3}\f{\Gamma(\f16+\f{2p}3)\Gamma(\f23)}{\Gamma(\f56)\Gamma(\f{2p}3)}=\f{\Gamma(\f23+\f{p}3)^2\Gamma(\f16+\f{2p}3)\Gamma(\f23)\Gamma(\f16+\f{p}3)\Gamma(1)}{\Gamma(\f13+\f{2p}3)\Gamma(\f{2p}3)\Gamma(\f56)\Gamma(\f12)\Gamma(\f23+\f{p}3)}\\
&\eq\f{\Gamma_p(\f23+\f{p}3)\Gamma_p(\f16+\f{2p}3)\Gamma_p(\f23)\Gamma_p(\f16+\f{p}3)}{\Gamma_p(\f13+\f{2p}3)\Gamma_p(\f{2p}3)\Gamma_p(\f56)\Gamma_p(\f12)}\\
&\eq\f{\Gamma_p(\f23)^2\Gamma_p(\f16)^2}{\Gamma_p(\f13)\Gamma_p(\f56)\Gamma_p(\f12)}\l(1-\f{p}3H_{(p-1)/3}+pH_{(p-1)/6}\r)\\
&\eq\f{\Gamma_p(\f23)^2\Gamma_p(\f16)^2}{\Gamma_p(\f13)\Gamma_p(\f56)\Gamma_p(\f12)}(1-2p\,q_p(2)-p\,q_p(3))\pmod{p^2}.
\end{align*}
In view of Lemma \ref{Gauss}, we arrive at
\begin{align*}
\f{\Gamma_p(\f23)^2\Gamma_p(\f16)^2}{\Gamma_p(\f13)\Gamma_p(\f56)\Gamma_p(\f12)}=\f{\Gamma_p(\f12)\Gamma_p(\f13)^2}{\Gamma_p(\f13)\Gamma_p(\f56)}=\f{\Gamma_p(\f13)^2}{\Gamma_p(\f23)}=-\Gamma_p\l(\f13\r)^3.
\end{align*}
So we have
\begin{equation}\label{sigma2}
\sigma_2\eq-\Gamma_p\l(\f13\r)^3(1-2p\,q_p(2)-p\,q_p(3))\pmod{p^2}.
\end{equation}
Substituting \eqref{sigma1} and \eqref{sigma2} into \eqref{sigma12} we obtain
$$
\sum_{k=0}^{p-1}\f{(\f13)_k(\f23)_k}{(1)_k^2}\l(\f98\r)^k\eq-\Gamma_p\l(\f13\r)^3\pmod{p^2}.
$$
Thus it suffices to show that
$$
\binom{(2p-2)/3}{(p-1)/3}\eq-\Gamma_p\l(\f13\r)^3\pmod{p^2}.
$$
In fact, it is routine to verify that
\begin{align*}
\binom{(2p-2)/3}{(p-1)/3}=-\f{\Gamma_p(\f13+\f{2p}3)}{\Gamma_p(\f23+\f{p}3)^2}&\eq-\Gamma_p\l(\f13\r)^3\l(1+\f{2p}3\l(H_{(2p-2)/3}-H_{(p-1)/3}\r)\r)\\
&\eq-\Gamma_p\l(\f13\r)^3\pmod{p^2}.
\end{align*}
This proves Theorem \ref{main1} in the case $p\eq1\pmod{3}$.

Below we assume that $p\eq2\pmod{3}$. Similarly,
\begin{align*}
&{}_2F_1\bigg[\begin{matrix}\f{1-2p}{3}&\f{2-2p}{3}\\&1-2p\end{matrix}\bigg|\ \f98\bigg]_{p-1}
\\&\qquad\eq\sum_{k=0}^{p-1}\f{(\f13)_k(\f23)_k}{(1)_k^2}\l(\f98\r)^k\l(1-2p\sum_{j=0}^{k-1}\f{1}{3j+1}-2p\sum_{j=0}^{k-1}\f{1}{3j+2}+2pH_k\r)\pmod{p^2}
\end{align*}
and
\begin{align*}
&{}_2F_1\bigg[\begin{matrix}\f{2-p}{3}&\f{1-p}{3}\\&1-p\end{matrix}\bigg|\ \f98\bigg]_{p-1}
\\&\qquad\eq\sum_{k=0}^{p-1}\f{(\f13)_k(\f23)_k}{(1)_k^2}\l(\f98\r)^k\l(1-p\sum_{j=0}^{k-1}\f{1}{3j+1}-p\sum_{j=0}^{k-1}\f{1}{3j+2}+pH_k\r)\pmod{p^2}.
\end{align*}
Combining the above two congruences and letting $n=(2p-1)/3$ in \eqref{lem1eq1} and $n=(p-2)/3$ in \eqref{lem1eq2} we deduce that
\begin{equation}\label{sigma34}
\sum_{k=0}^{p-1}\f{(\f13)_k(\f23)_k}{(1)_k^2}\l(\f98\r)^k\eq\sigma_3-\sigma_4\pmod{p^2},
\end{equation}
where
$$
\sigma_3:=2^{(5-p)/3}\f{(\f56)_{(p-2)/3}}{(\f23)_{(p-2)/3}}\quad \t{and}\quad \sigma_4:=2^{(1-2p)/3}\f{(\f12)_{(2p-1)/3}}{(\f13)_{(2p-1)/3}}.
$$

We first consider $\sigma_3$ modulo $p^2$. Note that
$$
\f{(\f12)_{(2p-1)/3}}{(1)_{(2p-1)/3}}=\f{\binom{(4p-2)/3}{(2p-1)/3}}{4^{(2p-1)/3}}=4^{(1-2p)/3}\f{\Gamma(\f{4p+1}3)}{\Gamma(\f{2p+2}3)^2}.
$$
Therefore, in view of \eqref{padicgamma2} we have
\begin{align*}
\sigma_3&=4\times2^{p-1}\times4^{(1-2p)/3}\f{\Gamma(\f16+\f{p}3)\Gamma(\f23)}{\Gamma(\f56)\Gamma(\f{p}3)}=4\times2^{p-1}\f{\Gamma(\f16+\f{2p}3)\Gamma(\f23+\f{2p}3)\Gamma(\f16+\f{p}3)\Gamma(\f23)}{\Gamma(\f12)\Gamma(\f13+\f{4p}3)\Gamma(\f56)\Gamma(\f{p}3)}\\
&\eq\f{p\Gamma_p(\f16)^2\Gamma_p(\f23)^2}{3\Gamma_p(\f12)\Gamma_p(\f13)\Gamma_p(\f56)}\pmod{p^2},
\end{align*}
where we note that $\Gamma(1/6+2p/3)/\Gamma(1/2), \Gamma(2/3+2p/3)/\Gamma(1/3+4p/3)$ and $\Gamma(1/6+p/3)/\Gamma(5/6)$ contain factors $p/2, 1/p$ and $p/6$, respectively. With the help of Lemma \ref{Gauss} and noting that $\Gamma_p(1/3)\Gamma_p(2/3)=(-1)^{2(p+1)/3}=1$, we obtain
$$
\f{\Gamma_p(\f16)^2\Gamma_p(\f23)^2}{\Gamma_p(\f12)\Gamma_p(\f13)\Gamma_p(\f56)}=\f{\Gamma_p(\f13)^2}{\Gamma_p(\f23)}=\Gamma_p\l(\f13\r)^3.
$$
Thus we get
\begin{equation}\label{sigma3}
\sigma_3\eq\f{p}{3}\Gamma_p\l(\f13\r)^3\pmod{p^2}.
\end{equation}
We now consider $\sigma_4$. Clearly,
\begin{equation}\label{sigma4}
\begin{aligned}
\sigma_4&=2^{1-2p}\binom{(4p-2)/3}{(2p-1)/3}\f{(1)_{(2p-1)/3}}{(\f13)_{(2p-1)/3}}=2^{1-2p}\f{\Gamma(\f13+\f{4p}3)\Gamma(\f13)}{\Gamma(\f23+\f{2p}3)\Gamma(\f{2p}3)}\\
&=2^{1-2p}\f{p\Gamma_p(\f13+\f{4p}3)\Gamma_p(\f13)}{3\Gamma_p(\f23+\f{2p}3)\Gamma_p(\f{2p}3)}\eq \f{p\Gamma_p(\f13)^2}{6\Gamma_p(\f23)}=\f{p}6\Gamma_p\l(\f13\r)^3\pmod{p^2}.
\end{aligned}
\end{equation}
Substituting \eqref{sigma3} and \eqref{sigma4} into \eqref{sigma34} we have
\begin{equation}\label{lhs}
\sum_{k=0}^{p-1}\f{(\f13)_k(\f23)_k}{(1)_k^2}\l(\f98\r)^k\eq\f{p}{6}\Gamma_p\l(\f13\r)^3\pmod{p^2}.
\end{equation}
On the other hand,
\begin{equation*}
\f{p}{\binom{(2p+2)/3}{(p+1)/3}}=-\f{p\Gamma_p(\f43+\f{p}3)^2}{\Gamma_p(\f53+\f{2p}3)}\eq\f{p}{6}\Gamma_p\l(\f13\r)^3\pmod{p^2}.
\end{equation*}
This, together with \eqref{lhs}, proves Theorem \ref{main1} in the case $p\eq2\pmod{3}$.

The proof of Theorem \ref{main1} is now complete.\qed

\section{Proof of Theorem \ref{main1'}}
\setcounter{lemma}{0}
\setcounter{theorem}{0}
\setcounter{equation}{0}
\setcounter{conjecture}{0}
\setcounter{remark}{0}
\setcounter{corollary}{0}

\begin{lemma}\label{lem3.1} Let $n$ be a nonnegative integer. Then
\begin{equation}\label{lem3.1eq1}
\sum_{k=0}^{n}(3n+k+2)\f{(-n)_k(\f13-n)_k}{(1)_k(-3n)_k}\l(\f98\r)^k=3\times2^{4/3-n}\f{\Gamma(\f23)\Gamma(\f76+n)}{\Gamma(\f12)\Gamma(\f13+n)}
\end{equation}
and
\begin{equation}\label{lem3.1eq2}
\sum_{k=0}^{n}(3n+k+3)\f{(-n)_k(-\f13-n)_k}{(1)_k(-3n-1)_k}\l(\f98\r)^k=3\times2^{1-n}\f{\Gamma(\f23)\Gamma(\f32+n)}{\Gamma(\f12)\Gamma(\f23+n)}.
\end{equation}
\end{lemma}
\begin{proof}
Denote the left-hand sides of \eqref{lem3.1eq1} and \eqref{lem3.1eq2} by $J(n)$ and $K(n)$ respectively. Via the Zeilberger algorithm, we find that
$$
(6n+7)J(n)-4(3n+1)J(n+1)=0
$$
and
$$
3(2n+3)K(n)-4(3n+2)K(n+1)=0.
$$
Then the identities can be checked by induction on $n$.
\end{proof}

The following lemma is the well-known Gauss multiplication formula whose $p$-adic analogue has been stated in Lemma \ref{Gauss}.
\begin{lemma}[{Robert \cite[p. 371]{Robert}}]\label{lem3.2}
For any $z\in\C$ and $m\in\{2,3,\ldots\}$, we have
\begin{equation}\label{lem3.2eq}
\prod_{0\leq j<m}\Gamma\l(z+\f{j}{m}\r)=(2\pi)^{(m-1)/2}m^{(1-2mz)/2}\Gamma(mz).
\end{equation}
\end{lemma}

\medskip
\noindent{\it Proof of Theorem \ref{main1'}}. Suppose that $p\eq1\pmod{3}$. As in the proof of Theorem \ref{main1}, by Lemma \ref{lem3.1} we can easily prove
\begin{equation}\label{sigma56}
\sum_{k=0}^{p-1}(k+1)\f{(\f13)_k(\f23)_k}{(1)_k^2}\l(\f98\r)^k\eq\sigma_5-\sigma_6\pmod{p^2},
\end{equation}
where
$$
\sigma_5:=6\times2^{(5-p)/3}\f{\Gamma(\f23)\Gamma(\f{2p+5}{6})}{\Gamma(\f12)\Gamma(\f{p}3)}\quad\t{and}\quad\sigma_6:=3\times2^{(5-2p)/3}\f{\Gamma(\f23)\Gamma(\f{4p+5}6)}{\Gamma(\f12)\Gamma(\f{2p}3)}.
$$
Note that
$$
-\f{\Gamma_p(\f16+\f{p}3)}{\Gamma_p(\f12)\Gamma_p(\f23+\f{p}3)}=\f{(\f12)_{(p-1)/3}}{(1)_{(p-1)/3}}=\f{\binom{(2p-2)/3}{(p-1)/3}}{4^{(p-1)/3}}=-2^{(2-2p)/3}\f{\Gamma_p(\f13+\f{2p}3)}{\Gamma_p(\f23+\f{p}3)^2}.
$$
Therefore, by Lemmas \ref{lem}, \ref{Gauss} and \ref{lem3.2} we obtain
\begin{equation}\label{sigma5}
\begin{aligned}
\sigma_5&=-2^{(11-2p)/3}\f{\Gamma(-\f13)\Gamma(\f{2p+5}6)}{\Gamma(\f{p}6)\Gamma(\f{p+3}6)}=-\f{2^{(8-2p)/3}p}3\f{\Gamma_p(-\f13)\Gamma_p(\f56+\f{p}3)}{\Gamma_p(\f{p}6)\Gamma_p(\f12+\f{p}6)}\\
&\eq-\f{4p\Gamma_p(\f16)\Gamma_p(\f23)^2\Gamma_p(\f56)}{\Gamma_p(\f12)^2\Gamma_p(\f13)}=-\f{4p}{\Gamma_p\l(\f13\r)^3}\pmod{p^2}.
\end{aligned}
\end{equation}
Also,
\begin{equation}\label{sigma6}
\begin{aligned}
\sigma_6=&-2^{(5-2p)/3}\f{\Gamma(\f{-1}3)\Gamma(\f56+\f{2p}3)}{\Gamma(\f12)\Gamma(\f{2p}3)}=-\f{2^{(2-2p)/3}p\Gamma_p(\f{-1}3)\Gamma_p(\f56+\f{2p}3)}{\Gamma_p(\f12)\Gamma_p(\f{2p}3)}\\
\eq&-\f{3p\Gamma_p(\f16)\Gamma_p(\f23)^2\Gamma_p(\f56)}{\Gamma_p(\f12)^2\Gamma_(\f13)}=-\f{3p}{\Gamma_p(\f13)^3}\pmod{p^2}.
\end{aligned}
\end{equation}
Substituting \eqref{sigma5} and \eqref{sigma6} into \eqref{sigma56} we get
$$
\sum_{k=0}^{p-1}(k+1)\f{(\f13)_k(\f23)_k}{(1)_k^2}\l(\f98\r)^k\eq-\f{p}{\Gamma_p(\f13)^3}.
$$
On the other hand, from the proof of Theorem \ref{main1} we know
$$
\binom{(2p-2)/3}{(p-1)/3}\eq-\Gamma_p\l(\f13\r)^3\pmod{p^2}.
$$
This proves Theorem \ref{main1'} in the case $p\eq1\pmod{3}$.

Now we assume $p\eq2\pmod{3}$. By Lemma \ref{lem3.1} it is easy to check that
\begin{equation}\label{sigma78}
\sum_{k=0}^{p-1}(k+1)\f{(\f13)_k(\f23)_k}{(1)_k^2}\l(\f98\r)^k\eq\sigma_7-\sigma_8\pmod{p^2},
\end{equation}
where
$$
\sigma_7:=12\times2^{(2-p)/3}\f{\Gamma(\f23)\Gamma(\f{2p+5}6)}{\Gamma(\f12)\Gamma(\f{p}3)}\quad\t{and}\quad\sigma_8:=3\times2^{(5-2p)/3}\f{\Gamma(\f23)\Gamma(\f{4p+5}6)}{\Gamma(\f12)\Gamma(\f{2p}3)}.
$$
Note that
$$
-\f{\Gamma_p(-\f16+\f{p}3)} {\Gamma_p(\f12)\Gamma_p(\f13+\f{p}3)}=\f{(\f12)_{(p-2)/3}}{(1)_{(p-2)/3}}=\f{\binom{(2p-4)/3}{(p-2)/3}}{4^{(p-2)/3}}=-2^{(4-2p)/3}\f{\Gamma_p(-\f13+\f{2p}3)}{\Gamma_p(\f13+\f{p}3)^2}.
$$
By Lemmas \ref{lem} and \ref{lem2} we have
\begin{equation}\label{sigma7}
\begin{aligned}
\sigma_7&=-4\times2^{(2-p)/3}\f{\Gamma_p(-\f13)\Gamma_p(\f56+\f{p}3)}{\Gamma_p(\f{p}3)\Gamma_p(\f12)}\\
&\eq-8\times2^{1-p}\f{\Gamma_p(-\f13+\f{2p}3)\Gamma_p(-\f13)\Gamma_p(\f56+\f{p}3)}{\Gamma_p(-\f16+\f{p}3)\Gamma_p(\f13+\f{p}3)\Gamma_p(\f{p}3)}\\
&\eq-12\times2^{1-p}\f{\Gamma_p(\f23)^2\Gamma_p(\f56)}{\Gamma_p(\f56)\Gamma_p(\f13)}\l(1+\f{p}3H_{(p-2)/3}\r)\\
&\eq-\f{12}{\Gamma_p(\f13)^3}\l(1-p\,q_p(2)-\f{p}{2}q_p(3)\r)\pmod{p^2}.
\end{aligned}
\end{equation}
Also,
\begin{equation}\label{sigma8}
\begin{aligned}
\sigma_8&=-2^{(5-4p)/3}\f{\Gamma(-\f13)\Gamma(\f56+\f{2p}3)}{\Gamma(\f{p}3)\Gamma(\f12+\f{p}3)}=-3\times2^{(5-4p)/3}\f{\Gamma_p(\f23)\Gamma_p(\f56+\f{2p}3)}{\Gamma_p(\f{p}3)\Gamma_p(\f12+\f{p}3)}\\
&\eq-\f{3\Gamma_p(-\f16+\f{p}3)^2\Gamma_p(\f13+\f{p}3)^2\Gamma_p(\f23)\Gamma_p(\f56+\f{2p}3)}{2\Gamma_p(\f12)^2\Gamma_p(-\f13+\f{2p}3)^2\Gamma_p(\f{p}3)\Gamma_p(\f12+\f{p}3)}\\
&\eq-\f{6}{\Gamma_p(\f13)^3}\l(1-2p\,q_p(2)-p\,q_p(3)\r)\pmod{p^2}.
\end{aligned}
\end{equation}
Combining \eqref{sigma78}--\eqref{sigma8} we have
$$
\sum_{k=0}^{p-1}(k+1)\f{(\f13)_k(\f23)_k}{(1)_k^2}\l(\f98\r)^k\eq-\f{6}{\Gamma_p(\f13)^3}\pmod{p^2}.
$$
On the other hand, it is routine to check that
$$
\binom{(2p+2)/3}{(p+1)/3}\eq6(1-p)\f{\Gamma_p(\f23+\f{2p}3)}{\Gamma_p(\f13+\f{p}3)^2}\eq\f{6(1-p)}{\Gamma_p(\f13)^3}\pmod{p^2}.
$$
Comparing the above two congruences we immediately obtain the desired result.

The proof of Theorem \ref{main1'} is now complete.\qed

\section{More results similar to Theorems \ref{48theorem}--\ref{main1'}}
\setcounter{lemma}{0}
\setcounter{theorem}{0}
\setcounter{equation}{0}
\setcounter{conjecture}{0}
\setcounter{remark}{0}
\setcounter{corollary}{0}

In this section, we list some congruences that can also be proved by some strange identities. However, we only give the outlines of their proofs since the proofs are quite similar to the ones of Theorems \ref{48theorem}--\ref{main1'}.

Sun \cite[Conjecture 5.14(\rm{i})]{Sun2011} posed the following conjecture as a variant of the third congruence in \eqref{mortenson's}.
\begin{conjecture}\label{sunconj}
Let $p>3$ be a prime. If $p\eq1\pmod{3}$ and $p=x^2+3y^2$ with $x\eq1\pmod{3}$, then we have
\begin{equation}\label{sunconjeq1}
\sum_{k=0}^{p-1}\f{\binom{2k}{k}\binom{4k}{2k}}{48^k}\eq2x-\f{p}{2x}\pmod{p^2}
\end{equation}
and
\begin{equation}\label{detx1}
\sum_{k=0}^{p-1}\f{k+1}{48^k}\binom{2k}{k}\binom{4k}{2k}\eq x\pmod{p^2}.
\end{equation}
If $p\eq2\pmod{3}$, then we have
\begin{equation}\label{sunconjeq2}
\sum_{k=0}^{p-1}\f{\binom{2k}{k}\binom{4k}{2k}}{48^k}\eq\f{3p}{2\binom{(p+1)/2}{(p+1)/6}}\pmod{p^2}.
\end{equation}
\end{conjecture}
\eqref{sunconjeq1} has been proved by G.-S. Mao and H. Pan \cite{MP} and its proof depends on the results concerning Legendre polynomials obtained by M. J. Coster and L. Van Hamme \cite{Cv}. In fact, we can confirm Conjecture \ref{sunconj} completely by using some strange hypergeometric identities.

\begin{theorem}\label{sec5main1}
For any prime $p>3$ we have
\begin{equation}\label{sec5main1eq1}
\sum_{k=0}^{p-1}\f{\binom{2k}{k}\binom{4k}{2k}}{48^k}\eq\begin{cases}\binom{(p-1)/2}{(p-1)/6}\l(1+\f{2p}3q_p(2)-\f{3p}4q_p(3)\r)\pmod{p^2}\ &{\rm if}\ p\eq1\pmod{3},\vspace{2mm}\\
3p/\big(2\binom{(p+1)/2}{(p+1)/6}\big)\pmod{p^2}\ &{\rm if}\ p\eq2\pmod{3},\end{cases}
\end{equation}
and
\begin{equation}\label{sec5main1eq2}
\begin{aligned}
&\sum_{k=0}^{p-1}\f{(2k+1)\binom{2k}{k}\binom{4k}{2k}}{48^k}\\
&\qquad\eq\begin{cases}p/\binom{(p-1)/2}{(p-1)/6}\pmod{p^2}\ &{\rm if}\ p\eq1\pmod{3},\vspace{2mm}\\
\binom{(p+1)/2}{(p+1)/6}\l(-\f23-\f{2p}3-\f{4p}9q_p(2)+\f{p}2q_p(3)\r)\pmod{p^2}\ &{\rm if}\ p\eq2\pmod{3}.\end{cases}
\end{aligned}
\end{equation}
\end{theorem}

\begin{remark}
It is known (cf. \cite[p. 283]{BEW}) that for any prime $p=x^2+3y^2\eq1\pmod{6}$ with $x\eq1\pmod{3}$ we have
$$
\binom{(p-1)/2}{(p-1)/6}\eq\l(2x-\f{p}{2x}\r)\l(1-\f{2p}3q_p(2)+\f{3p}4q_p(3)\r)\pmod{p^2}.
$$
This together with \eqref{sec5main1eq1} gives \eqref{sunconjeq1} and \eqref{detx1}.
\end{remark}

Applying \eqref{clausen} with $\alpha=1/4$ and $z=4/3$ and noting that $(1/4)_k(3/4)_k/(1)_k^2=\binom{4k}{2k}\binom{2k}{k}/64^k$ we arrive at
\begin{equation}\label{square'}
\l(\sum_{k=0}^{p-1}\f{\binom{4k}{2k}\binom{2k}{k}}{48^k}\r)^2\eq\sum_{k=0}^{p-1}\frac{\binom{2k}{k}^2\binom{4k}{k}}{(-144)^k}\pmod{p^2}.
\end{equation}
Therefore, combining \eqref{square'} with Theorem \ref{sec5main1} we can easily obtain the following result which was conjectured and partially proved by Z.-H. Sun (cf. \cite[Conjecture 2.2]{ZHSun2011} and \cite[Theorem 5.1]{ZHSun2013b}).

\begin{corollary}
Let $p>3$ be a prime. Then
$$
\sum_{k=0}^{p-1}\f{\binom{2k}{k}^2\binom{4k}{2k}}{(-144)^k}\eq\begin{cases}4x^2-2p\pmod{p^2}\ &{\rm if}\ p\eq1\pmod{3}\ \&\ p=x^2+3y^2\ (x,y\in\Z),\vspace{2mm}\\
 0\pmod{p^2}\ &{\rm if}\ p\eq2\pmod{3}.\end{cases}
$$
\end{corollary}

To show Theorem \ref{sec5main1}, we need the following identities which can be easily checked by Zeilberger's algorithm.

\begin{lemma}\label{sec5lem1}
Let $n$ be a nonnegative integer. Then we have the following identities:
\begin{gather}
\label{sec5lem1id1}\sum_{k=0}^{n}\f{(-n)_k(\f12-n)_k}{(1)_k(-4n)_k}\l(\f43\r)^k=\l(\f9{16}\r)^n\f{(\f23)_{2n}}{(\f12)_{2n}},\\
\label{sec5lem1id2}\sum_{k=0}^{n}\f{(-n)_k(-\f12-n)_k}{(1)_k(-4n-2)_k}\l(\f43\r)^k=\l(\f3{4}\r)^{2n+1}\f{(\f23)_{2n+1}}{(\f12)_{2n+1}},\\
\label{sec5lem1id3}\sum_{k=0}^{n}\f{(2n+k+1)(-n)_k(\f12-n)_k}{(1)_k(-4n)_k}\l(\f43\r)^k=\f{3^{2n+1}(\f13)_{2n}}{16^n(\f12)_{2n}},\\
\label{sec5lem1id4}\sum_{k=0}^{n}\f{(2n+k+2)(-n)_k(-\f12-n)_k}{(1)_k(-4n-2)_k}\l(\f43\r)^k=\f{9^{n+1}(\f13)_{2n+2}}{4^{2n+1}(\f12)_{2n+1}}.
\end{gather}
\end{lemma}

\medskip
\noindent{\it Proof of Theorem \ref{sec5main1}}. We should divide the proof into four cases that $p\eq1,5,7,11\pmod{12}$. We only prove \eqref{sec5main1eq1} for $p\eq1\pmod{12}$ briefly since \eqref{sec5main1eq1} in the other cases can be handled similarly.

By Lemma \ref{sec5lem1}, it is easy to check that
$$
\sum_{k=0}^{p-1}\f{\binom{2k}{k}\binom{4k}{2k}}{48^k}\eq\f32\l(\f9{16}\r)^{(p-1)/4}\f{(\f23)_{(p-1)/2}}
{(\f12)_{(p-1)/2}}-\f12\l(\f34\r)^{(3p-1)/2}\f{(\f23)_{(3p-1)/2}}{(\f12)_{(3p-1)/2}}\pmod{p^2}.
$$
From \cite[Lemma 4.1]{MT} we know that for any positive integer $n$ and integer $a$ not divisible by $p$ we have
\begin{equation}\label{MT's}
a^{(p-1)/2}\eq\l(\f{a}{p}\r)\sum_{k=0}^{n-1}\binom{\f12}{k}(p\,q_p(a))^k\pmod{p^n}.
\end{equation}
Thus we have
\begin{align*}
\f{3}{2}\l(\f34\r)^{(p-1)/2}\f{(\f23)_{(p-1)/2}}{(\f12)_{(p-1)/2}}
&\eq\f{3}{2}\l(1+\f12p\,q_p(3)-p\,q_p(2)\r)\f{\Gamma_p(\f16+\f{p}2)\Gamma_p(\f12)}{\Gamma_p(\f23)\Gamma_p(\f{p}2)}\\
&\eq-\f32\l(1-2p\,q_p(2)-\f{p}4q_p(3)\r)\Gamma_p\l(\f13\r)^3\pmod{p^2}
\end{align*}
and
\begin{align*}
\f12\l(\f34\r)^{(3p-1)/2}\f{(\f23)_{(3p-1)/2}}{(\f12)_{(3p-1)/2}}&\eq\f38\l(1+\f{3p}2q_p(3)-3p\,q_p(2)\r)\f{\f{2p}3\Gamma_p(\f16+\f{3p}2)\Gamma_p(\f12)}{\f{p}2\Gamma_p(\f23)\Gamma_p(\f{3p}2)}\\
&\eq-\f12\l(1-6p\,q_p(2)-\f{3p}4q_p(3)\r)\Gamma_p\l(\f13\r)^3\pmod{p^2}.
\end{align*}
Thus we have
$$
\sum_{k=0}^{p-1}\f{\binom{2k}{k}\binom{4k}{2k}}{48^k}\eq-\Gamma_p\l(\f13\r)^3\pmod{p^2}.
$$
On the other hand, it is routine to check that
$$
\binom{(p-1)/2}{(p-1)/6}\l(1+\f{2p}3q_p(2)-\f{3p}4q_p(3)\r)\eq-\Gamma_p\l(\f13\r)^3\pmod{p^2}.
$$
This completes the proof.\qed

Sun \cite[Conjecture 5.14(\rm{iii})]{Sun2011} made the following conjecture.
\begin{conjecture} For any prime $p>3$, if $p\eq1\pmod{4}$ and $p=x^2+4y^2\ (x,y\in\Z)$ with $x\eq1\pmod{4}$, then
\begin{equation}\label{sunconjeq3}
\sum_{k=0}^{p-1}\f{\binom{2k}{k}\binom{4k}{2k}}{72^k}\eq\l(\f6{p}\r)\l(2x-\f{p}{2x}\r)\pmod{p^2}
\end{equation}
and
\begin{equation}\label{detx2}
\sum_{k=0}^{p-1}\f{(1-k)\binom{2k}{k}\binom{4k}{2k}}{72^k}\eq\l(\f6{p}\r)x\pmod{p^2};
\end{equation}
if $p\eq3\pmod{4}$, then
\begin{equation}\label{sunconjeq4}
\sum_{k=0}^{p-1}\f{\binom{2k}{k}\binom{4k}{2k}}{72^k}\eq\l(\f6{p}\r)\f{2p}{3\binom{(p+1)/2}{(p+1)/4}}\pmod{p^2}.
\end{equation}
\end{conjecture}

We shall prove these congruences by establish the following result.
\begin{theorem}\label{sec5main2} For any prime $p>3$ we have
\begin{equation}\label{sec5main2eq1}
\sum_{k=0}^{p-1}\f{\binom{2k}{k}\binom{4k}{2k}}{72^k}\eq\begin{cases}(\f6{p})\binom{(p-1)/2}{(p-1)/4}\l(1-\f{p}2q_p(2)\r)\pmod{p^2}\ &{\rm if}\ p\eq1\pmod{4},\vspace{2mm}\\
2p(\f6{p})/\big(3\binom{(p+1)/2}{(p+1)/4}\big)\pmod{p^2}\ &{\rm if}\ p\eq3\pmod{4},\end{cases}
\end{equation}
and
\begin{equation}\label{sec5main2eq2}
\begin{aligned}
&\sum_{k=0}^{p-1}\f{(2k-1)\binom{2k}{k}\binom{4k}{2k}}{72^k}\\
&\qquad\eq\begin{cases}-p(\f6{p})/\binom{(p-1)/2}{(p-1)/4}\pmod{p^2}\ &{\rm if}\ p\eq1\pmod{4},\vspace{2mm}\\
(\f6{p})\binom{(p+1)/2}{(p+1)/4}\l(\f32+\f{3p}2-\f{3p}4q_p(2)\r)\pmod{p^2}\ &{\rm if}\ p\eq3\pmod{4}.\end{cases}
\end{aligned}
\end{equation}
\end{theorem}

\begin{remark}
From \cite[p. 281]{BEW} we know for any prime $p>3$, if $p\eq1\pmod{4}$ and $p=x^2+4y^2$ with $x\eq1\pmod{4}$ we have
$$
\binom{(p-1)/2}{(p-1)/4}\eq\l(2x-\f1{2x}\r)\l(1+\f{p}2q_p(2)\r)\pmod{p^2}.
$$
Combining this with Theorem \ref{sec5main2} we obtain \eqref{sunconjeq3} and \eqref{detx2}.
\end{remark}

Applying \eqref{clausen} with $\alpha=1/4$ and $z=8/9$ we arrive at
\begin{equation}\label{square''}
\l(\sum_{k=0}^{p-1}\f{\binom{4k}{2k}\binom{2k}{k}}{72^k}\r)^2\eq\sum_{k=0}^{p-1}\frac{\binom{2k}{k}^2\binom{4k}{k}}{648^k}\pmod{p^2}.
\end{equation}
Therefore, combining \eqref{square''} with Theorem \ref{sec5main2} we can easily obtain the following result which was conjectured and partially proved by Z.-H. Sun (cf. \cite[Conjecture 2.1]{ZHSun2011} and \cite[Theorem 5.1]{ZHSun2013b}).

\begin{corollary}
Let $p>3$ be a prime. Then
\begin{align*}
&\sum_{k=0}^{p-1}\f{\binom{2k}{k}^2\binom{4k}{2k}}{648^k}\\
&\qquad\eq\begin{cases}4x^2-2p\pmod{p^2}\ &{\rm if}\ p\eq1\pmod{4}\ \&\ p=x^2+4y^2\ with\ x,y\in\Z,\vspace{2mm}\\
 0\pmod{p^2}\ &{\rm if}\ p\eq3\pmod{4}.\end{cases}
\end{align*}
\end{corollary}

To prove Theorem \ref{sec5main2} we need the following preliminary result which can also be showed by Zeilberger's algorithm.

\begin{lemma}\label{sec5lem2}
Let $n$ be a nonnegative integer. Then
\begin{gather}
\label{sec5lem2id1}\sum_{k=0}^{n}\f{(-n)_k(\f12-n)_k}{(1)_k(-4n)_k}\l(\f89\r)^k=\f{\Gamma(\f12)\Gamma(3n)}{3^{2n-1}\Gamma(2n+\f12)\Gamma(n)},\\
\label{sec5lem2id2}\sum_{k=0}^{n}\f{(-n)_k(-\f12-n)_k}{(1)_k(-4n-2)_k}\l(\f89\r)^k=\f{3^{n+1}\Gamma(n+\f56)\Gamma(n+\f76)}{2\Gamma(\f12)\Gamma(2n+\f32)},\\
\label{sec5lem2id3}\sum_{k=0}^{n}\f{(10n-k+3)(-n)_k(\f12-n)_k}{(1)_k(-4n)_k}\l(\f89\r)^k=\f{3^{n+2}\Gamma(n+\f56)\Gamma(n+\f76)}{\Gamma(\f12)\Gamma(2n+\f12)},\\
\label{sec5lem2id4}\sum_{k=0}^{n}\f{(10n-k+8)(-n)_k(-\f12-n)_k}{(1)_k(-4n-2)_k}\l(\f89\r)^k=\f{2\Gamma(\f12)\Gamma(3n+3)}{9^{n}\Gamma(2n+\f32)\Gamma(n+1)}.
\end{gather}
\end{lemma}

\medskip
\noindent{\it Proof of Theorem \ref{sec5main2}}. We only prove \eqref{sec5main2eq1} for $p\eq1\pmod{4}$ since \eqref{sec5main2eq1} in the case $p\eq3\pmod{4}$ can be deduced in a similar way.

By Lemma \ref{sec5lem2} we can easily check that
$$
\sum_{k=0}^{p-1}\f{\binom{2k}{k}\binom{4k}{2k}}{72^k}\eq\f{\Gamma(\f12)\Gamma(\f{3p-3}{4})}{2\times3^{(p-1)/2}\Gamma(\f{p}2)\Gamma(\f{p-1}{4})}-\f{3^{(3p+1)/4}\Gamma(\f{9p+1}{12})\Gamma(\f{9p+5}{12})}{4\Gamma(\f12)\Gamma(\f{3p}2)}\pmod{p^2}.
$$
Clearly, by Lemma \ref{lem2} and \eqref{MT's} we have
\begin{align*}
&\f{\Gamma(\f12)\Gamma(\f{3p-3}{4})}{2\times3^{(p-1)/2}\Gamma(\f{p}2)\Gamma(\f{p-1}{4})}=\f{\Gamma_p(\f12)\Gamma_p(\f{3p-3}{4})}{2\times3^{(p-1)/2}\Gamma_p(\f{p}2)\Gamma_p(\f{p-1}{4})}\\
\eq&\ \l(\f{3}{p}\r)\l(\f12-\f{3p}4q_p(3)\r)\l(1+\f{3p}4H_{(p+3)/4}-\f{p}{4}H_{(p-5)/4}\r)\f{\Gamma_p(\f12)\Gamma_p(-\f34)}{\Gamma_p(-\f14)}\\
\eq&\ (-1)^{(p+3)/4}\l(\f{3}{p}\r)\l(\f12-\f{p}{4}q_p(3)-\f{3p}4q_p(2)\r)\Gamma_p\l(\f12\r)\Gamma_p\l(\f14\r)^2\pmod{p^2}.
\end{align*}
Moreover, by Lemma \ref{lem3.2} we have
\begin{align*}
&\f{3^{(3p+1)/4}\Gamma(\f{9p+1}{12})\Gamma(\f{9p+5}{12})\Gamma(\f{3p+3}{4})}{4\Gamma(\f12)\Gamma(\f{3p}2)\Gamma(\f{3p+3}{4})}=\f{3^{(2-6p)/4}\Gamma(\f12)\Gamma(\f14+\f{9p}4)}{2\Gamma(\f{3p}2)\Gamma(\f34+\f{3p}4)}=\f{3^{(2-6p)/4}\times\f{3p}2\Gamma_p(\f12)\Gamma_p(\f14+\f{9p}4)}{2\times\f{p}2\Gamma_p(\f{3p}2)\Gamma_p(\f34+\f{3p}4)}\\
\eq&\ (-1)^{(p+3)/4}\l(\f{3}{p}\r)\l(\f12-\f{3p}4q_p(3)\r)\l(1+\f{3p}2H_{(p-1)/4}\r)\Gamma_p\l(\f12\r)\Gamma_p\l(\f14\r)^2\\
\eq&\ (-1)^{(p+3)/4}\l(\f{3}{p}\r)\l(\f32-\f{3p}4q_p(3)-\f{9p}4q_p(2)\r)\Gamma_p\l(\f12\r)\Gamma_p\l(\f14\r)^2\pmod{p^2}.
\end{align*}
Thus we obtain
$$
\sum_{k=0}^{p-1}\f{\binom{2k}{k}\binom{4k}{2k}}{72^k}\eq\l(\f{6}{p}\r)\Gamma_p\l(\f12\r)\Gamma_p\l(\f14\r)^2\pmod{p^2},
$$
since
$$
\l(\f2{p}\r)\l(\f3{p}\r)=\l(\f6{p}\r)\ \t{and}\ \l(\f2{p}\r)=(-1)^{(p-1)^2/8-(p^2-1)/8}=(-1)^{(p-1)/4}.
$$
On the other hand, one may easily check that
$$
\binom{(p-1)/2}{(p-1)/4}\l(1-\f{p}2q_p(2)\r)\eq\Gamma_p\l(\f12\r)\Gamma_p\l(\f14\r)^2\pmod{p^2}.
$$
This completes the proof.\qed

\begin{Acks}
The authors are grateful to the anonymous referee for helpful comments. The work is supported by the National Natural Science Foundation of China (grant no. 11971222).
\end{Acks}


\begin{thebibliography}{99}
\bibitem{AAR} G. E. Andrews, R. Askey and R. Roy, {\it Special Functions}, Encyclopedia Math. Appl.
71, Cambridge University Press, Cambridge, 1999.

\bibitem{BEW} B. C. Berndt, R. J. Evans and K. S. Williams, {\it Gauss and Jacobi Sums}, John
Wiley \& Sons, 1998.

\bibitem{C} M. J. Coster, {\it Generalisation of a congruence of Gauss}, J. Number Theory {\bf29} (1988), no. 3, 300--310.

\bibitem{Cv}  M. J. Coster and L. Van Hamme, {\it Supercongruences of Atkin and Swinnerton-Dyer type for Legendre
polynomials}, J. Number Theory {\bf38} (1991), 265--286

\bibitem{Ekhad} S. B. Ekhad, {\it Forty ``strange" computer-discovered [and computer-proved (of course!)]
hypergeometric series evaluations}, The Personal Journal of Ekhad and Zeilberger, preprint, 2004. \url{https://sites.math.rutgers.edu/~zeilberg/mamarim/mamarimhtml/strange.html}

\bibitem{Guo} V. J. W. Guo, {\it Some generalizations of a supercongruence of van Hamme}, Integral Transforms Spec. Funct. {\bf28} (2017), no. 12, 888--899.

\bibitem{GuoLiu} V. J. W. Guo and J.-C. Liu, {\it Some congruences related to a congruence of Van Hamme}, Integral Transforms Spec. Funct. {\bf31} (2020), no. 3, 221--231.

\bibitem{L} E. Lehmer, {\it On congruences involving Bernoulli numbers and the quotients of Fermat and
Wilson}, Ann. of Math. {\bf39} (1938), 350--36

\bibitem{Liu} J.-C. Liu, {\it Some supercongruences on truncated ${}_3F_2$ hypergeometric series}, J. Difference Equ. Appl. {\bf24} (2018), no. 3, 438--451.

\bibitem{Liu2021} J.-C. Liu, {\it Supercongruences arising from transformations of hypergeometric series}, J. Math. Anal. Appl. {\bf497} (2021), Art. 124915.

\bibitem{LR} L. Long and R. Ramakrishna, {\it Some supercongruences occurring in truncated hypergeometric series}, Adv.
Math. {\bf290} (2016), 773--808.

\bibitem{MP} G.-S. Mao and H. Pan, {\it $p$-Adic analogues of hypergeometric identities}, preprint, arXiv:1703.01215v4.

\bibitem{MT} S. Mattarei and R. Tauraso, {\it Congruences for central binomial sums and finite polylogarithms}, J. Number Theory {\bf133} (2013), 131--157.

\bibitem{Morita} Y. Morita, {\it A $p$-adic analogue of the $\Gamma$-function}. J. Fac. Sci. Univ. Tokyo Sect. IA Math. {\bf22} (1975), no. 2, 255--266.

\bibitem{Mortenson1} E. Mortenson, {\it A supercongruence conjecture of Rodriguez-Villegas for a certain truncated hypergeometric
function}, J. Number Theory {\bf99} (2003), 139--147.

\bibitem{Mortenson2}  E. Mortenson, {\it Supercongruences between truncated ${}_2F_1$ hypergeometric functions and their Gaussian
analogs}, Trans. Amer. Math. Soc. {\bf355} (2003), 987--1007.

\bibitem{PWZ} M. Petkov\v{s}ek, H.S. Wilf and D. Zeilberger, $A = B$, A K Peters, Wellesley, 1996.

\bibitem{Robert} A. M. Robert, {\it A Course in $p$-Adic Analysis}, Graduate Texts in Mathematics, 198. Springer-Verlag, New
York, 2000.

\bibitem{RV}  F. Rodriguez-Villegas, {\it Hypergeometric families of Calabi-Yau manifolds}, in: Calabi-Yau Varieties and
Mirror Symmetry (Toronto, ON, 2001), Fields Inst. Commun., 38, Amer. Math. Soc., Providence, RI,
2003, pp. 223--231.

\bibitem {S} C. Schneider, {\it Symbolic summation assists combinatorics}, S\'{e}m. Lothar. Combin. {\bf56} (2007), Article B56b.

\bibitem{ZHSun2011} Z.-H. Sun, {\it Congruences concerning Legendre polynomials}, Proc. Amer. Math. Soc. {\bf139} (2011), no. 6, 1915--1929.

\bibitem{ZHSun2013a} Z.-H. Sun, {\it Congruences involving $\binom{2k}{k}^2\binom{3k}{k}$}, J. Number Theory {\bf133} (2013), no. 5, 1572--1595.

\bibitem{ZHSun2013b} Z.-H. Sun, {\it Congruences concerning Legendre polynomials \rm{II}}, J. Number Theory {\bf133} (2013), no. 6, 1950--1976.

\bibitem{ZHSun2013c} Z.-H. Sun, {\it Legendre polynomials and supercongruences}, Acta Arith. {\bf159} (2013), no. 2, 169--200.

\bibitem{ZHSun2014} Z.-H. Sun, {\it Generalized Legendre polynomials and related supercongruences}, J. Number Theory {\bf143} (2014),
293--319.

\bibitem{Sun2011} Z.-W. Sun, {\it Super congruences and Euler numbers}, Sci. China Math. {\bf54} (2011), 2509--2535.

\bibitem{Sun2013} Z.-W. Sun, {\it Supecongruences involving products of two binomial coefficients}, Finite Fields Appl. {\bf22} (2013),
24--44.

\bibitem{Sun2014} Z.-W. Sun, {\it On sums related to central binomial and trinomial coefficients}, in: M. B. Nathanson (ed.), Combinatorial and Additive Number Theory: CANT 2011 and 2012, Springer Proc. in Math. \& Stat., Vol. 101, Springer, New York, 2014, pp. 257--312.

\bibitem{Sun2015} Z.-W. Sun, {\it New series for some special values of $L$-functions}, Nanjing Univ. J. Math. Biquarterly {\bf 32} (2015), no. 2, 189--218.

\bibitem{Sun2019} Z.-W. Sun, {\it Open conjectures on congruences}, Nanjing Univ. J. Math. Biquarterly {\bf36} (2019), no. 1, 1--99.

\bibitem{Wang} C. Wang, {\it Symbolic summation methods and hypergeometric supercongruences}, J. Math. Anal. Appl. {\bf488} (2020), no. 1, Art. 124068.

\bibitem{WangPan} C. Wang and H. Pan, {\it Supercongruences concerning truncated hypergeometric series}, Math. Z. (2021). https://doi.org/10.1007/s00209-021-02772-0

\bibitem{W} J. Wolstenwholme, {\it On certain properties of prime numbers}, Quart. J. Appl. Math {\bf5} (1862), 35--39.
\end{thebibliography}
\end{document}